\documentclass[12pt,amscd]{amsart}
\usepackage[all]{xy}
\usepackage{graphicx}
\usepackage{amsmath,amsxtra,amssymb,latexsym, amscd,amsthm}

\usepackage{tikz}
\usepackage{ytableau}

 \usepackage{indentfirst}
\usepackage[mathscr]{eucal}
  \usepackage[pagebackref=true]{hyperref}

\newtheorem{thm}{Theorem}[section]
\newtheorem{cor}[thm]{Corollary}
\newtheorem{lem}[thm]{Lemma}

\theoremstyle{definition}
\newtheorem{defn}[thm]{Definition}

\newtheorem{quest}[thm]{Question}

\numberwithin{equation}{section}

\DeclareMathOperator{\Ass}{Ass}

\def\k {\mathrm{k}}

\begin{document}

\title{V-numbers of symbolic powers of cover ideals of graphs}

\author{Thanh Vu}
\address{Institute of Mathematics, VAST, 18 Hoang Quoc Viet, Hanoi, Vietnam}
\email{vuqthanh@gmail.com}


\date{}

\dedicatory{Dedicated to Professor Le Tuan Hoa on the occasion of his 70th birthday}
\commby{}
%-----------------------------------------------------------
% -----------------------------------------------------------

\begin{abstract} Let $G$ be a simple graph with cover ideal $J(G)$ in a polynomial ring $S$ with $|V(G)|$ variables. We prove that the local $v$-number of the symbolic powers $J(G)^{(t)}$ is linear for all $t \ge 1$ when $G$ is bipartite, and quasi-linear with period two for all $t \ge 2$ when $G$ is non-bipartite. Furthermore, we provide explicit formulas for these invariants in terms of the combinatorial data of $G$.
\end{abstract}

\maketitle
\section{Introduction}
\label{sect_intro}
The $v$-number of an ideal was introduced by Cooper, Seceleanu, Tohaneanu, Vaz Pinto, and Villarreal \cite{CSTVV}. Since then, it has attracted considerable interest \cite{BM,BMS,FM,KNS,S}. Conca \cite{C} and, independently, Ficarra and Sgroi \cite{FS}, proved that the $v$-function of the powers of an ideal is eventually linear. Kumar, Nanduri, and Saha \cite{KNS} proved that the $v$-function of the symbolic powers of a monomial ideal is eventually quasi-linear.

In this work, we show that for cover ideals of graphs, the local $v$-number is linear for all $t\ge1$ when the graph is bipartite, and is quasi-linear of period $2$ for all $t\ge2$ when the graph is non-bipartite. The key idea of the proof is to use a recent interpretation of Chau, Ha, Jayanthan, and Vu \cite{CHJV}, which translates the problem of computing the $v$-numbers of symbolic powers of squarefree monomial ideals into the study of almost covers. We then establish a decomposition result for almost covers of graphs and use it to deduce our main results. We now introduce these concepts in more detail.

Let $S = \k[x_1, \ldots, x_n]$ be a standard graded polynomial ring over a field $\k$, and let $I$ be a non-zero homogeneous ideal of $S$. For an associated prime $P$ of $I$, the local $v$-number of $I$ at $P$ is defined by 
$$v_P(I) = \min \{d \ge 0 \mid \exists f \in S_d \text{ such that } I : f = P\}.$$
The $v$-number of $I$ is defined by
\[
v(I)=\min_{P\in\Ass(I)}v_P(I),
\]
where $\Ass(I)$ denotes the set of associated primes of $I$.

Let $S=\k[x_1,\dots,x_n]$ be a polynomial ring over a field $\k$, and let $G$ be a simple graph with vertex set $V(G)=[n]=\{1,\ldots,n\}$. Denote by $J(G)$ the cover ideal of $G$. 
\begin{defn}
Let $G$ be a simple graph, and let $P=\{i,j\}$ be an edge of $G$. 
\begin{enumerate}
    \item A vector $\mathbf{a}=(a_1,\ldots,a_n)\in\mathbb{N}^n$ is called a $t$-cover of $G$ if $a_u+a_v\ge t$ for every edge $\{u,v\}\in E(G)$.
    \item A $t$-cover $\mathbf{a}$ is called a $P$-tight $t$-cover if $a_i+a_j=t$.
    \item A vector $\mathbf{a}\in\mathbb{N}^n$ is called a $P$-almost $t$-cover if
    \[
    a_i+a_j=t-1 \quad \text{and} \quad a_u+a_v\ge t
    \]
    for every other edge $\{u,v\}\in E(G)$.
\end{enumerate}
\end{defn}

\begin{defn}
Let $G$ be a simple graph, and let $P = \{i,j\}$ be an edge of $G$. For every positive integer $t$, we define 
$$\gamma_{t,P} = \min \{ |\mathbf{b}| \mid \mathbf{b} \text{ is a } P\text{-tight } t\text{-cover of } G \}.$$
\end{defn}

Our main results are as follows.

\begin{thm}\label{thm_local_bipartite} 
Let $G$ be a bipartite graph, and let $P = \{i,j\}$ be an edge of $G$. We also denote by $P$ the ideal $(x_i,x_j)$ of $S$. Then for all $t \ge 1$, we have 
$$v_P(J(G)^{(t)}) = v_P(J(G)) + (t-1)\gamma_{1,P}.$$
\end{thm}

\begin{thm}\label{thm_local_nonbipartite} 
Let $G$ be a simple graph, and let $P = \{i,j\}$ be an edge of $G$. We also denote by $P$ the ideal $(x_i,x_j)$ of $S$. Then for all $t \ge 2$, we have 
$$v_P(J(G)^{(t+2)}) = v_P(J(G)^{(t)}) + \gamma_{2,P}.$$
\end{thm}

In the next section, we review the necessary notation and prove Theorems~\ref{thm_local_bipartite} and \ref{thm_local_nonbipartite}. We then analyze the global $v$-number and provide explicit formulas for the initial values and slopes of these (quasi-)linear functions in terms of the combinatorial data of $G$.

\section{Local $v$-number of symbolic powers of cover ideals}

Throughout the paper, let $S=k[x_1,\ldots,x_n]$ be a standard graded polynomial ring over a field $k$. For a nonzero homogeneous ideal $I$ of $S$, we denote by $\Ass(I)$ the set of associated primes of $I$.

\subsection{Local $v$-numbers via integer programming}

For a monomial $f$, we denote by $\deg(f)$ its total degree and, for each $i\in[n]$, by $\deg_i(f)$ the exponent of $x_i$ in $f$. For a monomial prime ideal $P$, we further define
\[
\deg_P(f)=\sum_{x_i\in P}\deg_i(f).
\]

The following result, which follows from \cite{CHJV}, allows us to compute the local $v$-numbers of symbolic powers of squarefree monomial ideals via integer programming. We include a proof for completeness.

\begin{lem}\label{lem_criterion}
Let $I\subseteq S$ be a nonzero squarefree monomial ideal, let $P$ be an associated prime of $I$, and let $t\ge1$ be an integer. Let $u$ be a nonzero monomial in $S$. Then
\[
I^{(t)}:u=P
\]
if and only if
\[
\deg_P(u)=t-1
\quad\text{and}\quad
\deg_Q(u)\ge t
\]
for every associated prime $Q\neq P$ of $I$.
\end{lem}

\begin{proof}
First, assume that
\[
I^{(t)} : u = P.
\]
Localizing this equality at $P$ yields
\[
P^t S_P : u = P S_P.
\]
Because the variables outside $P$ become units in $S_P$, it follows that
\[
\deg_P(u) = t - 1.
\]

Now, let $Q \neq P$ be another associated prime of $I$. Since $P$ and $Q$ are incomparable, there exists a variable $x \in P \setminus Q$. Since $xu \in I^{(t)} \subseteq Q^{(t)}$ and $x \notin Q$, it follows that $u \in Q^{(t)}$. In particular,
\[
\deg_Q(u) \ge t.
\]

Conversely, assume that $u$ is a monomial such that
\[
\deg_P(u) = t - 1 \quad \text{and} \quad \deg_Q(u) \ge t
\]
for every associated prime $Q \neq P$ of $I$. Then $u \notin P^t$, and hence $u \notin I^{(t)}$. 

Now, let $x \in P$ be a variable. Since $\deg_P(xu) \ge t$, we have $xu \in P^t S_P \cap S = P^t$. Moreover, since $u \in Q^t$ for every associated prime $Q \neq P$ of $I$, we also have $xu \in Q^t$ for every such $Q$. Therefore, $xu \in I^{(t)}$. Finally, let $f$ be a monomial with $f \notin P$. Then $\deg_P(uf) = t - 1$, and hence $uf \notin P^t$, which implies $uf \notin I^{(t)}$. Thus, $I^{(t)} : u = P$.
\end{proof}

\subsection{Graphs and their cover ideals}  Let $G$ be a simple graph with vertex set $V(G) = \{1, \ldots, n\}$ and edge set $E(G)$. Throughout, we assume that $G$ has no isolated vertices.

\begin{defn}\label{definition1}
Let \( G \) be a simple graph with vertex set \( V(G) = \{1, \ldots, n\} \) and edge set \( E(G) \).

\begin{enumerate}
    \item A simple graph \( H \) is a subgraph of \( G \) if \( V(H) \subseteq V(G) \) and \( E(H) \subseteq E(G) \). It is an \emph{induced subgraph} of \( G \) if \( E(H) = E(G) \cap \big( V(H) \times V(H) \big) \).
    
    \item For a subset \( U \subseteq V(G) \), we denote by \( G[U] \) and \( G - U \) the induced subgraphs of \( G \) on \( U \) and on \( V(G) \setminus U \), respectively.
    
    \item For a vertex $v \in V(G)$, we denote by $N_G(v)$ the set of neighbors of $v$. For a subset $U \subseteq V(G)$, we denote by $N_G(U) = \bigcup_{u \in U} N_G(u)$ the set of all neighbors of vertices in $U$.
    
    \item A path \( P_n \) on \( n \) vertices is the graph with vertex set \( V(P_n) = \{1, \ldots, n\} \) and edge set
    \[
    E(P_n) = \{ \{1,2\}, \ldots, \{n-1,n\} \}.
    \]
    
    \item A cycle \( C_n \) on \( n \) vertices is the graph with vertex set \( V(C_n) = \{1, \ldots, n\} \) and edge set
    \[
    E(C_n) = E(P_n) \cup \{ \{1,n\} \}.
    \]
    
    \item A \emph{forest} is a graph with no cycles. A \emph{tree} is a connected forest.

   \item    A subset $W \subseteq V(G)$ is called a \emph{vertex cover} of $G$ if $W \cap e \neq \emptyset$ for every edge $e \in E(G)$. It is called a \emph{minimal vertex cover} if no proper subset of $W$ is a vertex cover of $G$. We denote by $\tau(G)$ the minimum size of a vertex cover of $G$.
   
   \item The \emph{cover ideal} of $G$, denoted by $J(G)$, is defined by
$$J(G) = \bigcap_{\{i,j\} \in E(G)} (x_i,x_j).$$

\end{enumerate}
\end{defn}

It is well known that
$$J(G) = (x_W \mid W \text{ is a minimal vertex cover of } G),$$
where $x_W = \prod_{i \in W} x_i$. The $t$-th symbolic power of $J(G)$ is given by
$$J(G)^{(t)}=\bigcap_{ \{i,j\} \in E(G)}(x_i, x_j)^t.$$

By Lemma~\ref{lem_criterion}, we obtain the following interpretation of the
local $v$-number of $J(G)^{(t)}$.

\begin{lem}\label{lem_int_prog}
Let $G$ be a simple graph, and let $J(G)$ denote its cover ideal. Let
$P=\{i,j\}$ be an edge of $G$, and also write $P=(x_i,x_j)$ for the
corresponding prime ideal of $S$. Then, for every integer $t\ge1$,
\[
v_P\bigl(J(G)^{(t)}\bigr)
=
\min\left\{
|\mathbf{a}|
\;\middle|\;
\mathbf{a}\text{ is a $P$-almost $t$-cover of $G$}
\right\}.
\]
\end{lem}
\begin{proof}
Let $f = x_1^{a_1}\cdots x_n^{a_n}$ be a monomial such that $J(G)^{(t)} : f = P$. By Lemma~\ref{lem_criterion}, the exponent vector $\mathbf{a} = (a_1, \ldots, a_n)$ is a $P$-almost $t$-cover of $G$. The conclusion then follows directly from the definition of the local $v$-number $v_P(J(G)^{(t)})$.
\end{proof}

\subsection{Local $v$-numbers of bipartite graphs}

\begin{lem}\label{lem_decomp1} 
Let $G$ be a bipartite graph, and let $P = \{i,j\}$ be an edge of $G$. Assume that $t \ge 1$ is an integer. Let $\mathbf{c} \in \mathbb{N}^n$ be a $P$-almost $(t+1)$-cover of $G$. Then there exists a $1$-cover $\mathbf{b}$ and a $P$-almost $t$-cover $\mathbf{a}$ of $G$ such that $\mathbf{c} = \mathbf{a} + \mathbf{b}$.
\end{lem}

\begin{proof}
Let $V(G) = U \cup V$ be a bipartition of $G$, and assume without loss of generality that $i \in U$ and $j \in V$. Choose $r \in \{1, \ldots, t+1\} \setminus \{c_i + 1\}$. Define $\mathbf{b}$ on $U$ and $V$ as follows:
\[
b_u = \begin{cases} 
1 & \text{if } u \in U \text{ and } c_u \ge r, \\
0 & \text{otherwise,} 
\end{cases}
\]
and
\[
b_v = \begin{cases} 
1 & \text{if } v \in V \text{ and } c_v \ge t + 2 - r, \\
0 & \text{otherwise.} 
\end{cases}
\]
We now prove that $\mathbf{b}$ is a $1$-cover of $G$.

Indeed, suppose for the sake of contradiction that there exists an edge $\{u,v\} \in E(G)$ with $u \in U$ and $v \in V$ such that $b_u = 0$ and $b_v = 0$. By definition, we have $c_u \le r - 1$ and $c_v \le t + 1 - r$. Consequently, $c_u + c_v \le t$. Since $\mathbf{c}$ is a $P$-almost $(t+1)$-cover, $c_u + c_v \ge t+1$ for all edges $\{u,v\} \neq \{i,j\}$. Thus, $c_u + c_v \le t$ is possible if and only if $u = i$ and $v = j$. But then $c_i \le r - 1$ and $c_j \le t + 1 - r$. Combined with $c_i + c_j = t$, this forces $c_i = r - 1$, which contradicts our choice of $r \neq c_i + 1$. Hence, no such edge exists, and $\mathbf{b}$ is a $1$-cover of $G$.

Now, let $\mathbf{a} = \mathbf{c} - \mathbf{b}$. We show that $\mathbf{a}$ is a $P$-almost $t$-cover of $G$. From the construction, exactly one of $b_i, b_j$ is equal to $1$, so $b_i + b_j = 1$. Thus, $a_i + a_j = (c_i + c_j) - (b_i + b_j) = t - 1$. 

For any other edge $\{u,v\} \neq \{i,j\}$ in $G$:
\begin{itemize}
    \item If $b_u + b_v = 2$, then $c_u + c_v \ge (r) + (t + 2 - r) = t + 2$, so $a_u + a_v = (c_u + c_v) - 2 \ge t$.
    \item If $b_u + b_v = 1$, then since $c_u + c_v \ge t + 1$, we have $a_u + a_v = (c_u + c_v) - 1 \ge t$.
\end{itemize}
Thus, $\mathbf{a}$ is a $P$-almost $t$-cover of $G$, completing the proof.
\end{proof}

\begin{proof}[Proof of Theorem~\ref{thm_local_bipartite}]
Let $\mathbf{b}$ be a $P$-tight $1$-cover of $G$ such that $|\mathbf{b}| = \gamma_{1,P}$. By Lemma~\ref{lem_int_prog}, there exists a $P$-almost $t$-cover $\mathbf{a}$ of $G$ such that $v_P(J(G)^{(t)}) = |\mathbf{a}|$. The sum $\mathbf{a} + \mathbf{b}$ is then a $P$-almost $(t+1)$-cover of $G$. Thus, by Lemma~\ref{lem_int_prog}, 
$$v_P(J(G)^{(t+1)}) \le |\mathbf{a}| + |\mathbf{b}| = v_P(J(G)^{(t)}) + \gamma_{1,P}.$$

Conversely, let $\mathbf{c}$ be a $P$-almost $(t+1)$-cover of $G$ such that $|\mathbf{c}| = v_P(J(G)^{(t+1)})$. By Lemma~\ref{lem_decomp1}, we can write $\mathbf{c} = \mathbf{a} + \mathbf{b}$, where $\mathbf{b}$ is a $P$-tight $1$-cover of $G$ and $\mathbf{a}$ is a $P$-almost $t$-cover of $G$. By Lemma~\ref{lem_int_prog}, 
$$v_P(J(G)^{(t+1)}) = |\mathbf{a}| + |\mathbf{b}| \ge v_P(J(G)^{(t)}) + \gamma_{1,P}.$$
Since the two inequalities yield equality for every $t \ge 1$, the conclusion follows by induction.
\end{proof}

Hence, the local $v$-number of $J(G)^{(t)}$ at $P$ is completely determined by $v_P(J(G))$ and $\gamma_{1,P}$. We now give formulas for these values in terms of graph invariants of $G$.

\begin{lem}\label{lem_v_1}
Let $G$ be a simple graph and let $P = \{i,j\}$ be an edge of $G$. We also denote by $P = (x_i, x_j)$ the corresponding prime ideal in $S$. Then,
\[
v_P(J(G)) = \deg_G(i) + \deg_G(j) - 2 - |N_G(i) \cap N_G(j)| + \tau(G - N_G[\{i,j\}]).
\]
\end{lem}

\begin{proof}
Let $\mathbf{a}$ be a $P$-almost $1$-cover of $G$. By definition, we must have $a_i = a_j = 0$. Consequently, $a_u \ge 1$ for all $u \in N_G[\{i,j\}] \setminus \{i,j\}$. The only uncovered edges remain in $G - N_G[\{i,j\}]$. The conclusion then follows directly from Lemma~\ref{lem_criterion}.
\end{proof}

\begin{lem}\label{lem_gamma_1}
Let $G$ be a simple graph and let $P = \{i,j\}$ be an edge of $G$. We also denote by $P = (x_i, x_j)$ the corresponding prime ideal in $S$. Then,
\[
\gamma_{1,P}(G) = \min \left\{ \deg_G(i) + \tau(G - N_G[i]), \, \deg_G(j) + \tau(G - N_G[j]) \right\}.
\]
\end{lem}

\begin{proof}
Let $\mathbf{a}$ be a $P$-tight $1$-cover of $G$. By definition, we must have either $a_i = 1$ and $a_j = 0$, or $a_i = 0$ and $a_j = 1$. 

If $a_i = 1$ and $a_j = 0$, then $a_u \ge 1$ for all $u \in N_G(j)$. The only uncovered edges remaining are those in $G - N_G[j]$. Thus, the minimal weight of such a vector is $\deg_G(j) + \tau(G - N_G[j])$. By symmetry, if $a_i = 0$ and $a_j = 1$, the minimal weight is $\deg_G(i) + \tau(G - N_G[i])$. Taking the minimum of these two cases yields the desired conclusion.
\end{proof}

\begin{cor}
Let $P_n$ be a path graph on $n$ vertices. Then for all $t \ge 1$,
\[
v(J(P_n)^{(t)}) = \left\lfloor \frac{n-1}{2} \right\rfloor + (t-1) \left\lfloor \frac{n}{2} \right\rfloor.
\]
\end{cor}

\begin{proof}
By Lemma~\ref{lem_v_1}, the local $v$-number at $P$ for $t = 1$ is given by
\[
v_P(J(P_n)) = \begin{cases}
1 + \left\lfloor \frac{n-3}{2} \right\rfloor & \text{if } P = (x_1, x_2) \text{ or } P = (x_{n-1}, x_n), \\[6pt]
2 + \left\lfloor \frac{i-2}{2} \right\rfloor + \left\lfloor \frac{n-i-2}{2} \right\rfloor & \text{if } P = (x_i, x_{i+1}) \text{ with } 2 \le i \le n-2.
\end{cases}
\]

By Lemma~\ref{lem_gamma_1}, if $P = (x_1, x_2)$ or $P = (x_{n-1}, x_n)$, we have
\[
\gamma_{1,P}(P_n) = \min \left\{ 1 + \left\lfloor \frac{n-2}{2} \right\rfloor, \, 2 + \left\lfloor \frac{n-3}{2} \right\rfloor \right\} = \left\lfloor \frac{n}{2} \right\rfloor.
\]
When $P = (x_i, x_{i+1})$ with $2 \le i \le n-2$, we have
\[
\gamma_{1,P}(P_n) = \min \left\{ 2 + \left\lfloor \frac{i-1}{2} \right\rfloor + \left\lfloor \frac{n-i-2}{2} \right\rfloor, \, 2 + \left\lfloor \frac{i-2}{2} \right\rfloor + \left\lfloor \frac{n-i-1}{2} \right\rfloor \right\} = \left\lfloor \frac{n}{2} \right\rfloor.
\]

Thus, the local $v$-numbers $v_P(J(P_n)^{(t)}) = v_P(J(P_n)) + (t-1)\gamma_{1,P}(P_n)$ are linear functions of $t$ with the same leading coefficient $\left\lfloor \frac{n}{2} \right\rfloor$. Taking the minimum over all associated primes $P \in \operatorname{Ass}(S/J(P_n))$ yields the desired global $v$-number formula.
\end{proof}

\begin{cor}
Let $C_{2m}$ be an even cycle on $2m$ vertices, where $m \ge 2$. Then for all $t \ge 1$,
\[
v(J(C_{2m})^{(t)}) = tm.
\]
\end{cor}

\begin{proof}
By symmetry, all edges of $C_{2m}$ are equivalent. Let $P = (x_1, x_2)$. By Lemma~\ref{lem_v_1}, we have 
\[
v_P(J(C_{2m})) = 2 + (m - 2) = m.
\]
By Lemma~\ref{lem_gamma_1}, we also have $\gamma_{1,P}(C_{2m}) = m$. The conclusion follows.
\end{proof}

\subsection{Local $v$-numbers of non-bipartite graphs}

\begin{lem}\label{lem_decomp}
Let $G$ be a simple graph, let $P = \{i,j\}$ be an edge of $G$, and let $t \ge 2$ be an integer. If $\mathbf{c} \in \mathbb{N}^n$ is a $P$-almost $(t+2)$-cover of $G$, then there exist a $P$-tight $2$-cover $\mathbf{b}$ and a $P$-almost $t$-cover $\mathbf{a}$ of $G$ such that $\mathbf{c} = \mathbf{a} + \mathbf{b}$.
\end{lem}
\begin{proof}
We consider two cases. 
\medskip 

\noindent{\textbf{Case 1.}} Assume that $c_i, c_j > 0$. We follow an argument similar to that in \cite[Theorem 5.1]{HHT}. Let $A = \{u \in [n] \mid c_u = 0\}$ and $B = N_G(A)$. Since $c_i, c_j > 0$, we have $i, j \notin A$. Furthermore, $i, j \notin B$. Indeed, if $i \in B$, then there exists $u \in A$ such that $i \in N_G(u)$. Since $\mathbf{c}$ is a $P$-almost $(t+2)$-cover of $G$, we would have $c_i + c_u \ge t+2$. Since $c_u = 0$, this implies $c_i \ge t+2$, which contradicts the assumption that $c_i + c_j = t+1$. 

We claim that $c_v \ge t+2$ for all $v \in B$. Indeed, for any $v \in B$, there exists $u \in A$ such that $\{u,v\} \in E(G)$. Since $i, j \notin A$, this edge is distinct from $\{i,j\}$. Therefore, by the definition of a $P$-almost cover, we deduce that $c_u + c_v \ge t+2$. Since $c_u = 0$, it follows that $c_v \ge t+2$. In particular, we have $A \cap B = \emptyset$, and $A$ is an independent set of $G$.

Now, we define $\mathbf{b} \in \mathbb{N}^n$ as follows: $b_u = 0$ for $u \in A$, $b_v = 2$ for $v \in B$, and $b_k = 1$ for $k \in C = [n] \setminus (A \cup B)$. Clearly, $b_k \le c_k$ for all $k$ by the definitions of $\mathbf{b}, A,$ and $B$. Set $\mathbf{a} = \mathbf{c} - \mathbf{b}$. Since $i, j \in C$, by definition we have $b_i = b_j = 1$. We now prove that $\mathbf{a}$ is a $P$-almost $t$-cover of $G$. Indeed, $a_i + a_j = (c_i + c_j) - (b_i + b_j) = (t+1) - 2 = t-1$. 

Now, for any other edge $\{u,v\}$ of $G$, we consider the following cases:
\begin{itemize}
    \item If $u \in A$ and $v \in B$, then $c_u + c_v \ge t+2$ and $b_u + b_v = 2$; hence, $a_u + a_v \ge t$.
    \item If $u \in A$ and $v \in C$, then $b_u + b_v = 1$ and $c_u + c_v \ge t+1$; hence, $a_u + a_v \ge t$.
    \item If $u \in B$ and $v \in B$, then $b_u + b_v = 4$ and $c_u + c_v \ge (t+2) + (t+2)$; hence, $a_u + a_v \ge 2t \ge t$.
    \item If $u \in B$ and $v \in C$, then $b_u + b_v = 3$. Since $c_u \ge t+2$ and $c_v \ge 1$, we have $c_u + c_v \ge t+3$, and hence $a_u + a_v \ge t$.
    \item If $u \in C$ and $v \in C$, then $b_u + b_v = 2$; hence, $a_u + a_v = (c_u + c_v) - 2 \ge (t+2) - 2 = t$.
\end{itemize}
This proves that $\mathbf{a}$ is a $P$-almost $t$-cover of $G$.

\medskip 

\noindent\textbf{Case 2.} One of $c_i, c_j = 0$. Without loss of generality, assume that $c_i = 0$, which implies $c_j = t+1$. Let $T_c$ be the subgraph of $G$ consisting of edges $\{u,v\}$ such that $c_u + c_v = t+2$. We define the following sets of vertices of $G$ recursively as follows:
\[
Y_0 = \{ u \mid c_u = 0 \}, \quad X_0 = N_G(Y_0), \quad Y_{r+1} = Y_r \cup N_{T_c}(X_r), \quad X_{r+1} = X_r \cup N_G(Y_{r+1}).
\]
Finally, set $X = \bigcup_{r \ge 0} X_r$ and $Y = \bigcup_{r \ge 0} Y_r$.

Since $G$ is finite, these unions are finite. We now prove the following claim: if $x \in X$, then $c_x \ge t+1$; if $y \in Y$, then $c_y \le 1$. We prove this by induction on $r$ for each layer of $X$ and $Y$.

First, by definition, if $y \in Y_0$, then $c_y = 0$. For $x \in X_0$, there exists a neighbor $y_0 \in Y_0$; hence, $c_x + c_{y_0} \ge t+1$, which implies $c_x \ge t+1$. Now, assume that $y \in Y_{r+1} \setminus Y_r$. Then there exists $x_r \in X_r$ such that $\{x_r, y\} \in E(T_c)$. In particular, $c_{x_r} + c_y = t+2$. Since $c_{x_r} \ge t+1$ by the inductive hypothesis, it follows that $c_y \le 1$. Finally, assume that $x \in X_{r+1} \setminus X_r$. Then there exists $y_{r+1} \in Y_{r+1}$ such that $\{x, y_{r+1}\} \in E(G)$. If this edge is $\{i,j\}$, then $x = j$ and $c_x = t+1$. Otherwise, $c_x + c_{y_{r+1}} \ge t+2$. Since $c_{y_{r+1}} \le 1$, we obtain $c_x \ge t+1$.

In particular, $X \cap Y = \emptyset$. Let $R = V(G) \setminus (X \cup Y)$. We define $\mathbf{b}$ as follows:
\[
b_v = \begin{cases}
    2 & \text{if } v \in X, \\
    0 & \text{if } v \in Y, \\
    1 & \text{if } v \in R.
\end{cases}
\]
Since $c_x \ge t+1 \ge 2$ for $x \in X$ and $c_r \ge 1$ for all $r \notin Y$, we have $b_v \le c_v$ for all $v \in V(G)$. Set $\mathbf{a} = \mathbf{c} - \mathbf{b}$. We now prove that $\mathbf{b}$ is a $P$-tight $2$-cover and $\mathbf{a}$ is a $P$-almost $t$-cover. 

Indeed, we have $i \in Y$ and $j \in X$; hence, $b_i + b_j = 2$. It remains to show that for any other edge $\{u,v\}$ of $G$, we have $2 \le b_u + b_v \le c_u + c_v - t$. Since all neighbors of vertices in $Y$ belong to $X$, there are no edges between $Y$ and $Y$ or between $Y$ and $R$. Thus, we consider the following cases:

\begin{itemize}
    \item $u \in X, v \in Y$: Then $b_u + b_v = 2$, and since $\{u,v\} \neq \{i,j\}$, we have $c_u + c_v \ge t+2$.
    \item $u \in X, v \in R$: Since all $T_c$-neighbors of vertices in $X$ belong to $Y$, this edge is not in $T_c$; in other words, $c_u + c_v \ge t+3$. Here, we have $b_u + b_v = 3 \le c_u + c_v - t$.
    \item $u \in R, v \in R$: Then $b_u + b_v = 2$, and $c_u + c_v \ge t+2$.
    \item $u \in X, v \in X$: By our claim, $c_u, c_v \ge t+1$. Hence, $c_u + c_v \ge 2t+2 \ge t + 4$ since $t \ge 2$. In this case, $b_u + b_v = 4 \le c_u + c_v - t$.
\end{itemize}

This completes the proof of the lemma.
\end{proof}

\begin{proof}[Proof of Theorem~\ref{thm_local_nonbipartite}]
Let $\mathbf{b}$ be a $P$-tight $2$-cover of $G$ such that $|\mathbf{b}| = \gamma_{2,P}$. By Lemma~\ref{lem_int_prog}, there exists a $P$-almost $t$-cover $\mathbf{a}$ of $G$ such that $v_P(J(G)^{(t)}) = |\mathbf{a}|$. The sum $\mathbf{a} + \mathbf{b}$ is then a $P$-almost $(t+2)$-cover of $G$. Thus, by Lemma~\ref{lem_int_prog},
\[
v_P(J(G)^{(t+2)}) \le |\mathbf{a}| + |\mathbf{b}| = v_P(J(G)^{(t)}) + \gamma_{2,P}.
\]

Conversely, let $\mathbf{c}$ be a $P$-almost $(t+2)$-cover of $G$ such that $|\mathbf{c}| = v_P(J(G)^{(t+2)})$. By Lemma~\ref{lem_decomp}, we can write $\mathbf{c} = \mathbf{a} + \mathbf{b}$, where $\mathbf{b}$ is a $P$-tight $2$-cover and $\mathbf{a}$ is a $P$-almost $t$-cover of $G$. By Lemma~\ref{lem_int_prog},
\[
v_P(J(G)^{(t+2)}) = |\mathbf{a}| + |\mathbf{b}| \ge v_P(J(G)^{(t)}) + \gamma_{2,P}.
\]
The conclusion follows.
\end{proof}

\begin{cor}
Let $C_{2m+1}$ be an odd cycle on $2m+1$ vertices. Then for all $t \ge 1$,
\[
v(J(C_{2m+1})^{(t)}) = \begin{cases}
    (2m+1)q & \text{if } t = 2q, \\
    (2m+1)q + m & \text{if } t = 2q+1.
\end{cases}
\]
\end{cor}

\begin{proof}
By symmetry, all edges of $C_{2m+1}$ are equivalent. Let $P = (x_1, x_2)$ and set $n = 2m+1$. Let $\mathbf{a}$ be a $P$-almost $t$-cover of $C_n$. Then $a_1 + a_2 \ge t-1$, while $a_i + a_{i+1} \ge t$ for all $i \neq 1$. Summing over all edges of $C_n$, we deduce that $2|\mathbf{a}| \ge nt - 1$, and hence $|\mathbf{a}| \ge \left\lceil \frac{nt-1}{2} \right\rceil$. Furthermore, this optimal bound is achieved as follows:

\begin{itemize}
    \item When $t = 2q$, set $\mathbf{a} = (q-1, q, q, \ldots, q, q+1)$.
    \item When $t = 2q+1$, set $\mathbf{a} = (q, q, q+1, q, q+1, \ldots, q+1)$.
\end{itemize}
The conclusion follows.
\end{proof}

We now give a formula for $\gamma_{2,P}$, the common leading coefficient of the quasi-linear functions.

\begin{defn}
Let $P = \{i,j\}$ be an edge of $G$. An independent set $A$ of $G$ is called $P$-admissible if either $i \in A$ or $j \in A$, or $\{i,j\} \cap (A \cup N_G(A)) = \emptyset$.
\end{defn}

\begin{lem}\label{lem_2_P}
Let $G$ be a simple graph and let $P = \{i,j\}$ be an edge of $G$. Then
\[
\gamma_{2,P} = n - \max \{ |A| - |N_G(A)| \mid A \text{ is a } P\text{-admissible independent set} \}.
\]
\end{lem}

\begin{proof}
For a $P$-admissible independent set $A$, we define a vector $\mathbf{b}_A$ as follows:
\[
(\mathbf{b}_A)_v = \begin{cases} 
0 & \text{if } v \in A, \\ 
2 & \text{if } v \in N_G(A), \\ 
1 & \text{if } v \notin A \cup N_G(A). 
\end{cases}
\]
Then $\mathbf{b}_A$ is a $P$-tight $2$-cover of $G$. Hence, 
\[
\gamma_{2,P} \le 2|N_G(A)| + |V(G) \setminus (A \cup N_G(A))| = n - (|A| - |N_G(A)|).
\]

Conversely, let $\mathbf{b}$ be a $P$-tight $2$-cover of $G$. Then we have either $b_i = 0, b_j = 2$; $b_i = 2, b_j = 0$; or $b_i = b_j = 1$. Set $A = \{v \in V(G) \mid b_v = 0\}$. By symmetry, the first two cases are identical, so we consider two cases:

If $b_i = 0$ and $b_j = 2$, then $A$ is an independent set with $i \in A$; hence, $A$ is $P$-admissible. Clearly, $b_v \ge 2$ for all $v \in N_G(A)$, and by definition, $b_v \ge 1$ for all $v \notin A \cup N_G(A)$. Thus, the minimal value is obtained precisely when $\mathbf{b} = \mathbf{b}_A$. 

When $b_i = b_j = 1$, we have $\{i,j\} \cap (A \cup N_G(A)) = \emptyset$, so $A$ is again $P$-admissible. Once more, the minimum is achieved when $\mathbf{b} = \mathbf{b}_A$. The conclusion follows.
\end{proof}

\section{Global $v$-number of symbolic powers of cover ideals}

\begin{lem} \label{lem:2} 
Let $G$ be a bipartite graph on $n$ vertices. Then $v(J(G)^{(t)})$ becomes a linear function in $t$ for all $t \ge n-2$.
\end{lem}

\begin{proof} 
For each prime ideal $P = (x_i,x_j)$ corresponding to an edge $\{i,j\}$ of $G$, we set $f_P(t) = v_P(J(G)^{(t)})$. Note that $\tau(G) = \min_{P} \gamma_{1,P}$. Let $Q$ be an associated prime ideal such that $\gamma_{1,Q} = \tau(G)$ and $f_Q(1)$ is as small as possible among all such choices of $Q$.

Now, for any other associated prime $P$ where $\gamma_{1,P} > \gamma_{1,Q}$, we have 
\[
f_P(t) - f_Q(t) = (f_P(1) - f_Q(1)) + (t-1)(\gamma_{1,P} - \gamma_{1,Q}).
\]
By Lemma~\ref{lem_v_1}, $f_Q(1) \le n-2$. Furthermore, $f_P(1) \ge 1$. Consequently, this difference is non-negative for all $t \ge n-2$. The conclusion follows.
\end{proof}

We have not found any example where the $v$-function of powers of cover ideals of a bipartite graph is not linear for all $t \ge 1$. We thus pose the following question.

\begin{quest} Let $G$ be a bipartite graph. Is $v(J(G)^{t})$ a linear function of $t$ for all $t \ge 1$?
\end{quest}

\begin{lem} \label{lem:3} 
Let $G$ be a nonbipartite graph on $n$ vertices. Then $v(J(G)^{(t)})$ becomes a quasi-linear function in $t$ with period two for all $t \ge 4n-5$.
\end{lem}

\begin{proof} 
For each prime ideal $P = (x_i,x_j)$ corresponding to an edge $\{i,j\}$ of $G$, we set $f_P(t) = v_P(J(G)^{(t)})$. Let $Q$ be an associated prime ideal such that $\gamma_{2,Q} = \min_{P} \gamma_{2,P}$ and $f_Q(2)$ is as small as possible among all such choices of $Q$.

Now, for any other associated prime $P$ where $\gamma_{2,P} > \gamma_{2,Q}$, by Theorem~\ref{thm_local_nonbipartite}, we have 
\[
f_P(2t) - f_Q(2t) = (f_P(2) - f_Q(2)) + (t-1)(\gamma_{2,P} - \gamma_{2,Q}).
\]
We first claim that $f_P(2) \le 2n-3$. Indeed, set $b_i = 1, b_j = 0$, and $b_\ell = 2$ for all $\ell \notin \{i,j\}$. Then $\mathbf{b}$ is a $P$-tight $2$-cover of $G$. Hence, $f_P(2) \le |\mathbf{b}| = 2n-3$. Since $f_Q(2) \ge 1$, this difference is non-negative for all $t \ge 2n-3$, i.e., $2t \ge 4n-6$.

Similarly, we have $2 \le f_P(3) \le 2n-3$ for all $P$. Indeed, setting $b_i = 1, b_j = 2$, and $b_\ell = 2$ for all $\ell \notin \{i,j\}$ yields $f_P(3) \le 2n-3$. Thus, the quasi-linear functions evaluated at odd steps also dominate for all $t \ge 4n-5$. The conclusion follows.
\end{proof}

\vspace{0.2cm}
\noindent {\bf Data Availability} Data sharing is not applicable to this article as no datasets were generated or analyzed during the current study.
\vspace{0.2cm}

\noindent {\bf Conflict of interest} There are no competing interests of either financial or personal nature.


\begin{thebibliography}{2}




\bibitem[BM]{BM}
P.~Biswas and M.~Mandal,
A study of v-number for some monomial ideals, {\em Collect. Math. } {\bf 76} (2025).


\bibitem[BMS]{BMS}
P.~Biswas, M.~Mandal, and K.~Saha,
Asymptotic Behaviour and Stability Index of v-Numbers of Graded Ideals, {\em Vietnam J. Math.} {\bf 62} (2026), 13. 



\bibitem[CHJV]{CHJV} T. Chau, T. Ha, A.V. Jayanthan, and T. Vu, {\it Comparing v-numbers of symbolic and ordinary powers of squarefree monomial ideals},  preprint.

\bibitem[C]{C}
A.~Conca,
A note on the v-invariant, {\em Proc. Am. Math. Soc.} {\bf 152} (2024), 2349--2351.



\bibitem[CSTVV]{CSTVV}
S. M. Cooper, A. Seceleanu, S. O. Tohaneanu, M. Vaz Pinto and R. H. Villarreal, Generalized minimum distance functions and algebraic invariants of Geramita ideals,
{\em Adv. in Appl. Math.} {\bf 112} (2020), 101940.



\bibitem[FM]{FM}
A.~Ficarra and P.~M.~Marques,
The v-function of powers of sum of ideals, {\em J. Algebr. Comb.} {\bf 62} (2025), 13. 



\bibitem[FS]{FS}
A.~Ficarra and E.~Sgroi,
Asymptotic behaviour of the v-number of homogeneous ideals, {\em Journal of Algebra} {\bf 704} (2026), 273--297.

\bibitem [HHT]{HHT} J. Herzog, T. Hibi and N. V. Trung, {\it Symbolic powers of monomial ideals and vertex cover algebras},  Adv. Math. {\bf 210}(1) (2007), 304 - 322.


\bibitem[KNS]{KNS}
M.~Kumar, R.~Nanduri, and K.~Saha,
The slope of the v-function and the Waldschmidt constant, {\em Journal of Pure and Applied Algebra} {\bf 229} (2025), 107881.


\bibitem[S]{S}
K.~Saha,
The v-Number and Castelnuovo-Mumford Regularity of Cover Ideals of Graphs, {\em International Mathematics Research Notices} {\bf 11} (2024), 9010--9019.

\end{thebibliography}
\end{document}